\documentclass[10pt, a4paper]{amsart}
\usepackage{a4}
\usepackage{amsfonts, amssymb}
\usepackage[all]{xy}
\usepackage{graphicx}

\newtheorem{thm}{Theorem}[section]
\newtheorem{cor}[thm]{Corollary}

\newtheorem{lem}[thm]{Lemma}

\newtheorem{defn}[thm]{Definition}

\newtheorem{prop}[thm]{Proposition}

\numberwithin{equation}{section}

\def\P{\mathbf{P}}
\def\E{\mathbf{E}}
\def\V{\mathbf{V}}
\def\unogrande{\text{\large\bf 1}}
\def\uno{\text{\bf 1}}

\begin{document}
\title{Visible lattice points in random walks}

\author{Javier Cilleruelo}

\address{Instituto de Ciencias Matem\'aticas (ICMAT) and Departamento de Matem\'aticas,
Universidad Aut\'onoma  de Madrid, 28049 Madrid, Spain.}
\email{franciscojavier.cilleruelo@uam.es}

\author{Jos\'{e} L. Fern\'andez}
\address{Departamento de Matem\'aticas, Universidad Aut\'onoma  de Madrid, 28049 Madrid, Spain.}
\email{joseluis.fernandez@uam.es}

\author{Pablo Fern\'andez}
\address{Departamento de Matem\'aticas, Universidad Aut\'onoma  de Madrid, 28049 Madrid, Spain.}
\email{pablo.fernandez@uam.es}
\thanks{Javier Cilleruelo has been supported by MINECO project MTM2014-56350-P  and ICMAT	Severo Ochoa project SEV-2011-0087. Jos\'{e} L. Fern\'{a}ndez and Pablo Fern\'{a}ndez are partially supported by Fundaci\'{o}n Akusmatika.}

\begin{abstract}
We consider the possible visits to visible points of a random walker moving up and right in the integer lattice (with probability $\alpha$ and $1-\alpha$, respectively) and starting from the origin. 


We show that, almost surely, the asymptotic proportion of strings of $k$ consecutive \textit{visible} lattice points visited by such an $\alpha$-random walk is a certain constant $c_k(\alpha)$, which is actually an (explicitly calculable) polynomial in $\alpha$ of degree $2\lfloor(k-1)/2\rfloor $. For $k=1$, this gives that, almost surely, the asymptotic proportion of time the random walker is visible from the origin is $c_1(\alpha)=6/\pi^2$, independently of $\alpha$.
\end{abstract}



\maketitle

\section{Introduction}
Fix $\alpha\in (0,1)$ and consider an $\alpha$-random walk in the two-dimensional lattice, starting at $P_0=(0,0)$, and given, for $i\ge 0$, by
\begin{equation}\label{def alphaRW}
P_{i+1}=P_i+\begin{cases}(1,0) \text{ with probability }\alpha,
\\
(0,1) \text{ with probability }1-\alpha,\end{cases}
\end{equation}
where the steps are independent. Notice that only steps $(1,0)$ and $(0,1)$ are allowed.

We are interested in estimating the proportion of visible lattice points (and more generally, the proportion of strings of $k$ consecutive visible lattice points) visited by such an $\alpha$-random walk. Recall that $(a,b)$ is visible (from the origin) if and only if $\gcd(a,b)=1$.

Associated to the $\alpha$-random walk, consider the sequence $(X_i)_{i\ge 1}$ of Bernoulli random variables given by
$$
X_i=
\begin{cases}1,\quad\text{if } P_i \text{ is visible,}
\\ 0,\quad\text{if not,} \end{cases}
$$
and write
$$
\overline{S}_n=\frac{X_1+\cdots +X_n}n
$$
for the variable that registers the proportion of visible points visited by the $\alpha$-random walk in the first $n$ steps. Observe that the variables $X_i$ are not independent.

 Our first result reads:
\begin{thm}\label{main}For any $\alpha\in (0,1)$,
$$
\lim_{n\to \infty}\overline{S}_n=\frac{6}{\pi^2}
$$ almost surely.
\end{thm}

This result agrees with intuition, as Dirichlet's classical result claims that the probability that a random lattice point is visible is asymptotically $6/\pi^2$.
However, if instead of the set of visible points,  we take an arbitrary subset $\mathcal{B}$ in the lattice with positive density $\beta>0$, then the analogous to Theorem \ref{main} does not hold in general. In fact it is easy to construct a subset $\mathcal{B}$ of the lattice, with asymptotic density $1$, such that, for all $\alpha\in[0,1]$, the $\alpha$-random walk avoids $\mathcal{B}$ almost surely. On the other hand, if $\mathcal{B}$ is a subset of $\mathbb{N}$ with asymptotic density $\beta>0$, then the one-dimensional $\alpha$-random walk defined by $P_0=0$ and
$$
P_{i+1}=P_i+\begin{cases}1 \text{ with probability }\alpha,
\\
0 \text{ with probability }1-\alpha,\end{cases}
$$
stays on $\mathcal{B}$ an asymptotic proportion $\beta$ of the time.

The proof of Theorem~\ref{main} relies on number-theoretical estimates of the mean and the variance of~$\overline{S}_n$.
See Section~\ref{sec:2}.

Our main result concerns with the proportion of $k$ consecutive visible lattice points visited. Interestingly, this proportion depends on $\alpha$ for $k\ge 3$.

Define, for $k\ge 1$, the random variable
$$
\overline S_{n,k}=\frac{X_1\cdots X_k+\cdots + X_n\cdots X_{n+k-1}}{n}
$$
that registers the proportion of $k$ consecutive visible lattice points in the first $n+k-1$ steps in an $\alpha$-random walk.

\begin{thm}\label{Sk}For any $\alpha\in (0,1)$,
$$
\lim_{n\to \infty}\overline S_{n,k}=c_k(\alpha)\quad\text{almost surely,}
$$ 	
 where
$$
c_k(\alpha)=b_k(\alpha)\prod_{p\ge k}\Big(1-\frac k{p^2}\Big)
$$
and $b_k(\alpha)$ is a polynomial in $\alpha$ with rational coeficientes  and  degree $2\lfloor (k-1)/2\rfloor$ that can be explicitly calculated.
\end{thm}

The first cases of $b_k(\alpha)$ in Theorem \ref{Sk} are
$$
b_1(\alpha)=1,\quad b_2(\alpha)=1,\quad b_3(\alpha)=\frac{1-\alpha+\alpha^2}2,\quad b_4(\alpha)=\frac{6-13\alpha+13\alpha^2}{18}.
$$
For instance, for $\alpha=1/2$, the first probabilities above are $c_1(1/2)=6/\pi^2\approx 0.6079,\ c_2(1/2)\approx 0.3226,\ c_3(1/2)\approx 0.1882$, and $c_4(1/2)\approx 0.1041$.

The case $k=1$ in Theorem \ref{Sk} is actually Theorem \ref{main}, but the latter has to be proved separately, as it is used in the proof of the former (see Section~\ref{sec4}). The approach to prove Theorem \ref{Sk} is different from that used in Theorem~\ref{main}, and has the same flavour as that used in \cite{Ci}.

\smallskip
Theorem \ref{Sk} has some direct and amusing consequences.

 We say that a sequence of $k$ lattice points  $P_i,\dots,P_{i+k-1}$ is a run of exactly $k$ visible lattice points (in a sequence $P_0,P_1,P_2,\dots,$) if $P_i,\dots,P_{i+k-1}$ are visible but $P_{i-1}$ and $P_{i+k}$ are not visible. In other words, if $(1-X_{i-1})X_i\cdots X_{i+k-1}(1-X_k)=1$.
\begin{cor}
In an $\alpha$-random walk, $\alpha\in (0,1)$, the proportion of runs of exactly $k$ consecutive visible lattice points tends to $c_k(\alpha)-2c_{k+1}(\alpha)+c_{k+2}(\alpha)$ almost surely.
\end{cor}

We say that $P_i$ represents a change of visibility (in $P_0,P_1,P_2,\dots,$) if $P_{i-1}$ or $P_i$ are visible, but not both.  In other words, if $(X_{i-1}-X_i)^2=1$.
\begin{cor}In an $\alpha$-random walk with $\alpha\in (0,1)$, the proportion of changes of visibility tends to $2c_1(\alpha)-2c_2(\alpha)=12/\pi^2-2\prod_{p} (1-2/p^{2})$ almost surely.
	\end{cor}
	
We can associate, to each real number $x\in [0,1)$, expressed in binary form, an infinite walk, starting at $P_0=(0,0)$, by identifying each $1$ of its binary representation with the step $(1,0)$, and each $0$ with the step $(0,1)$. To avoid ambiguities we do not consider representations where all the digits are $1$ from some place onwards. Define now
$$
\overline{S}_n(x)=\frac{X_1+\cdots +X_n}n,
$$
 where $X_i=1$ if the the  lattice point $P_i$ in this walk is visible, and $X_i=0$ otherwise. Theorem~\ref{main} for $\alpha=1/2$ can be rephrased as:
 $$
 \lim_{n\to \infty}\overline S_n(x)=6/\pi^2
 $$
 for all $x\in[0,1)$ except for an exceptional  set of measure zero.

It is natural to ask wether rational numbers belong to this exceptional set or not. The following result settles the question for this special  case of walks with a periodic pattern.
\begin{thm}\label{main2}
If $x\in [0,1)$ is rational,  $ \lim_{n\to \infty}\overline S_n(x) $ is a rational number.
\end{thm}

Notice that, in particular, $ \lim_{n\to \infty}\overline S_n(x)\ne 6/\pi^2$ for $x$ rational. For example, $\lim_{n\to \infty}\overline S_n(x)={1}/{2}$ for $x=0.\overline{10}$, while $\lim_{n\to \infty}\overline S_n(x)=7/{12}$ for $x=0.10000\overline{10}$. See Section~\ref{sec:rational}.

\section{Some preliminary results}
 The estimates for the mean and the variance of $\overline{S}_n$ contained in the proof of Theorem \ref{main} rely on an estimate for the binomial theorem sum restricted to indices in a certain residue class.
	\begin{lem}\label{lema1}Fix integers $n\ge 1$ and $d\le n$, and let $r\in\{0,1,\dots, d-1\}$. Then, for any $\alpha\in(0,1)$,
		$$
		\sum_{l\equiv r{ \ \text{\rm mod}\ } d}\binom nl \alpha^l(1-\alpha)^{n-l}=\frac 1d+O\Big (\frac 1{\sqrt{\alpha(1-\alpha) n}}\Big ).$$
	\end{lem}
	\begin{proof} 
The sequence $({n\choose l} \alpha^l(1-\alpha)^{n-l})_l$ is 
unimodal, and in fact,
$$
\binom n{l-1} \alpha^{l-1}(1-\alpha)^{n-l+1}\le \binom n{l}\alpha^{l}(1-\alpha)^{n-l}\iff l\le \alpha (n+1).
$$

 Let $l_0=\left \lfloor \alpha (n+1)\right \rfloor.$
Then we have
$$
\binom n0(1-\alpha)^n<\binom n1\alpha(1-\alpha)^{n-1}<\dots \le \binom n{l_0} \alpha^{l_0}(1-\alpha)^{n-l_0}
$$
and
$$
\binom nn \alpha^n<\binom n{n-1}\alpha^{n-1}(1-\alpha)<\cdots <\binom n{l_0+1}\alpha^{l_0+1}(1-\alpha)^{n-l_0-1}.
$$
The maximum value of these numbers is $\binom n{l_0} \alpha^{l_0}(1-\alpha)^{n-l_0}$, and Stirling's formula implies that, for any $l$,
\begin{equation}
\label{stirling}
\binom nl\alpha^l(1-\alpha)^{n-l}
\le \binom{n}{l_0}\alpha^{l_0}(1-\alpha)^{n-l_0}
\ll \frac{1}{\sqrt{\alpha (1-\alpha) n}}.
\end{equation}

Denote by $j_0$  the largest integer such that $r+j_0d-1\le l_0$ and $j_1$ the largest integer such that $r+j_1d\le n$. Then,
$$
\binom n{r+md}\alpha^{r+md}(1-\alpha)^{n-r-md}\le \frac 1d\sum_{l=r+md}^{r+(m+1)d-1}\binom nl \alpha^l(1-\alpha)^{n-l}
$$
for $m=0,\dots,j_0-1$, and
$$
\binom n{r+md}\alpha^{r+md}(1-\alpha)^{n-r-md}\le \frac 1d\sum_{l=r+(m-1)d+1}^{r+md}\binom nl \alpha^l(1-\alpha)^{n-l}
$$
for $m=j_0+1,\dots,j_1$. Summing up these inequalities and adding the missing term corresponding to $m=j_0$ we have
\begin{align*}
\sum_{l\equiv r\ \text{(mod $d$)}}&\, \binom nl \alpha^l(1-\alpha)^{n-l}
\\
&\le
\frac 1d\sum_l\binom nl \alpha^l(1-\alpha)^{n-l}
+\binom n{r+j_0d}\alpha^{r+j_0d}(1-\alpha)^{n-r-j_0d}
\\
(\text{by \eqref{stirling}} )&\le  \frac 1d+O\Big(\frac 1{\sqrt{\alpha(1-\alpha)n}}\Big).	
\end{align*}
		
		The lower bound is obtained similarly.
	\end{proof}

Since $\alpha\in (0,1)$ will be fixed, we shall omit the dependence on $\alpha$ in the error terms throughout the paper.
\begin{lem}
\label{imp}
Let $\alpha\in(0,1/2]$. For any  integers $n\ge 1$, $a,b\ge 0$, we have
$$
\sum_{0\le l\le n}\binom nl\alpha^l(1-\alpha)^{n-l}\unogrande_{\gcd (l+a,n+b)=1}=\sum_{d\mid n+b}\frac{\mu(d)}d+O\Big(\frac{\tau(n+b)}{\sqrt{ n}}\Big),
$$	
where $\unogrande_A$ denotes the indicator function of the event $A$, and $\tau(m)$ stands for the divisor function.
\end{lem}
\begin{proof}
We have that $$\unogrande_{\gcd (l+a,n+b)=1}=\sum_{d\mid \gcd (l+a,n+b)}\mu(d),$$ which follows from the identity $\sum_{d\mid m}\mu(d)=1$ if $m=1$ and $0$ if $m>1$.
Thus,
\begin{align*}
&\sum_{0\le l\le n}\binom nl\alpha^l(1-\alpha)^{n-l}\unogrande_{\gcd (l+a,n+b)=1}
=\sum_{d\mid n+b}\mu(d)
\sum_{\substack{0\le l\le n\\l\equiv -a \ \text{(mod $d$)}}}\binom nl\alpha^l(1-\alpha)^{n-l}
\\
&\qquad =\sum_{d\mid n+b}\mu(d)\Big(\frac 1d+O\Big(\frac 1{\sqrt{ n}}\Big)  \Big)
=\sum_{d\mid n+b}\frac{\mu(d)}d+ O\Big(\frac{\tau(n+b)}{\sqrt{ n}}\Big),
\end{align*}
where Lemma \ref{lema1} was used in the second identity.\end{proof}

The following estimates involving the divisor function will be useful.
\begin{lem}\label{imp2} We have:
\begin{itemize}
		\item[i)]$\sum_{i\le n}\frac{\tau(i)}{\sqrt i}\ll \sqrt n \log n.$
		\item[ii)]$\sum_{1\le i<j\le n}\frac{\tau(j)}{\sqrt{j-i}}\ll n^{3/2}\log n.$
\end{itemize}
\end{lem}

\begin{proof} i) The partial sums of the divisor function satisfies the well-known estimate $D(n)=\sum_{m\le n}\tau(m)\sim n\log n$ (see \cite{HW}, Theorem 318). We use this, and summation by parts, to write		
\begin{align*}
\sum_{i\le n}\frac{\tau(i)}{\sqrt i}&= \sum_{i\le n}\frac{D(i)-D(i-1)}{\sqrt i}=\sum_{i\le n-1}D(i)\Big(\frac 1{\sqrt i}-\frac 1{\sqrt{i+1}}\Big)+\frac{D(n)}{\sqrt n}
\\
&=\sum_{i\le n-1}D(i)\, \Big(\frac 1{(\sqrt i+\sqrt{i+1})\sqrt i\sqrt{i+1}}\Big)+\frac{D(n)}{\sqrt n}
\\
&\le  \frac 12\sum_{i\le n-1}\frac{D(i)}{i^{3/2}}+\frac{D(n)}{\sqrt n}\ll\sum_{i\le n-1}\frac{\log i}{\sqrt i}+\sqrt n\log n\ll \sqrt n\log n.
\end{align*}
		
For ii),
$$
\sum_{1\le i<j\le n}\frac{\tau(j)}{\sqrt{j-i}}=\sum_{j\le n}\tau(j)\sum_{i<j}\frac 1{\sqrt{j-i}} \ll \sum_{j\le n}\tau(j)\sqrt j\le \sqrt n\sum_{j\le n}\tau(j)\ll n^{3/2}\log n.
$$
\end{proof}

\begin{lem}\label{imp3}
$$
\sum_{d\le n}\frac{\mu(d)}d\, \Big\lfloor \frac{n}{d}\Big\rfloor =\frac{6n}{\pi^2}+O(\log n).
$$
\end{lem}

\begin{proof}It is a consequence of the well-known formula $\sum_{d=1}^{\infty}{\mu(d)}/{d^2}=6/{\pi^2}$ (see \cite{HW}, Theorem 287):
\begin{align*}
\sum_{d\le n} \frac{\mu(d)}{d} \, \Big\lfloor \frac{n}{d}\Big\rfloor
&=\sum_{d\le n}\frac{\mu(d)}{d^2}-\sum_{d\le n}\frac{\mu(d)}{d}\, \Big\{\frac nd\Big\}
\\&= n\Big(\frac{6}{\pi^2}+O\Big(\sum_{d>n}\frac 1{d^2}\Big)\Big)	+O\Big(\sum_{d\le n}\frac 1d\Big)=\frac{6n}{\pi^2}+O(1)+O(\log n).	
\end{align*}
\end{proof}

\section{Proof of Theorem \ref{main}}\label{sec:2}

Theorem \ref{main} is a consequence of Propositions \ref{e1} and~\ref{var1}, and Lemma~\ref{Borel} below.

Using the notation for $\alpha$-random walks from the introduction, we can estimate:
\begin{prop}\label{e1}For any $\alpha\in (0,1)$ we have
$$
\E(X_1+\cdots +X_n)= \frac{6n}{\pi^2}+O(\sqrt{n}\,\log n\, ).
$$
\end{prop}
\begin{proof}
The coordinates of a lattice point $P_i$ sum $i$, and are of the form $(l,i-l)$, for some $l=0,\dots, i$, with probability $\binom il\alpha^l(1-\alpha)^{i-l}$. The lattice point $(l,i-l)$ is visible if and only if $\gcd(l,i)=1$. Thus,
$$
\E(X_i)=\sum_{l=0}^i\binom il \, \alpha^l\,(1-\alpha)^{i-l}\,\unogrande_{\gcd(l,i)=1}\ .
$$

Lemma \ref{imp}, with $n=i$ and $a=b=0$, gives
\begin{align*}
\E(X_i)=
\sum_{d\mid i}\frac{\mu(d)}d+O\Big (\frac{\tau(i)}{\sqrt{i}}\Big ).
\end{align*}

Finally, by Lemma \ref{imp2}\,i) and Lemma \ref{imp3}, we have
\begin{align*}
\E(X_1+\cdots +X_n)
&
=\sum_{i\le n}\Big (\sum_{d\mid i}\frac{\mu(d)}d+O\Big (\frac{\tau(i)}{\sqrt{i}}\Big )\Big )
=\sum_{d}\frac{\mu(d)}d\Big \lfloor \frac nd \Big \rfloor+O\Big (\sum_{i\le n}\frac{\tau(i)}{\sqrt{i}}   \Big )
\\
&=
\frac{6n}{\pi^2}+O( \sqrt {n} \log n  ),
\end{align*}
as required.
\end{proof}
\begin{prop}\label{var1}For any $\alpha\in (0,1)$ we have
$$
\V(X_1+\cdots +X_n)\ll n^{3/2}\log n.
$$
\end{prop}

\begin{proof} First, using Proposition \ref{e1},
\begin{align}\label{id}
\V(X_1+\cdots +X_n)&=\E\Big ( (X_1+\cdots+X_n)^2\Big )-\E(X_1+\cdots +X_n)^2
\\
&=\sum_{i,j\le n}\E(X_iX_j)-\Big (\frac{6n}{\pi^2}\Big )^2+O(n^{3/2}\log n).
\end{align}

Assume that $i<j$. The lattice points $P_i$ and $P_j$ will be of the form $P_i=(l,i-l),\ P_j=(l+r,j-l-r)$ for some $0\le l\le i$ and $0\le r\le j-i$ with probability
$$
\binom il\alpha^l(1-\alpha)^{i-l}\binom{j-i}{r}\alpha^{r}(1-\alpha)^{j-i-r}.
$$
Thus,
\begin{align*}
&\E(X_iX_j)
\\
&\quad =\sum_{0\le l\le i}\binom il\alpha^l(1-\alpha)^{i-l}\unogrande_{\gcd(i,l)=1}\sum_{0\le r\le j-i}\binom{j-i}{r}\alpha^{r}(1-\alpha)^{j-i-r} \unogrande_{\gcd(l+r,j)=1}.
\end{align*}

Now, using Lemma \ref{imp} with $n=j-i,\ a=l,\ b=i$ in the inner sum, and then with $n=i,\ a=b=0$ in the first sum, we get
\begin{align*}
\E(X_iX_j)&=\Big( \sum_{d\mid i}\frac{\mu(d)}d+  O\Big(\frac{\tau(i)}{\sqrt{ i}}\Big)    \Big)\, \Big( \sum_{d\mid j}\frac{\mu(d)}d+ O\Big (\frac{\tau(j)}{\sqrt{j-i}}\Big ) \Big).
\end{align*}

Notice that $\sum_{d\mid n}{\mu(d)}/{d}={\phi(n)}/{n}\le 1$, where $\phi(n)$ denotes Euler's totient function. Thus, for $i< j$ we have
$$
\E(X_iX_j)=\sum_{d\mid i}\frac{\mu(d)}d\sum_{d\mid j}\frac{\mu(d)}d+O\Big (\frac{\tau(j)}{\sqrt{j-i}}\Big )+O\Big (\frac{\tau(i)}{\sqrt{ i}}\Big ).
$$
Thus, by Lemma \ref{imp2} we have
\begin{align*}
	\sum_{i<j\le n}\E(X_iX_j)
=\sum_{i<j\le n}\sum_{d\mid i}\frac{\mu(d)}d\sum_{d\mid j}\frac{\mu(d)}d+O(n^{3/2}\log n  ).
	\end{align*}
Adding the diagonal terms we get
\begin{align}\label{al}
\sum_{i,j\le n}\E(X_iX_j)
=\sum_{i,j\le n}\sum_{d\mid i}\frac{\mu(d)}d\sum_{d\mid j}\frac{\mu(d)}d+O(n^{3/2}\log n).
\end{align}
The main term can be estimated using Lemma \ref{imp3}:
\begin{align*}
\sum_{i,j\le n}\sum_{d\mid i}\frac{\mu(d)}d\sum_{d\mid j}\frac{\mu(d)}d= \Big (\sum_{d}\frac{\mu(d)}d\Big \lfloor \frac nd\Big \rfloor \Big )^2
= \Big (\frac {6n}{\pi^2}\Big )^2+O(n\log n).
\end{align*}
Now we plug this into \eqref{al}, and then into \eqref{id}, and we are done.
\end{proof}

Finally, Theorem \ref{main}  follows by combining the estimates of Propositions~\ref{e1} and~\ref{var1} with the following standard result (see, for instance, Section 3.2 in \cite{Du}).

\begin{lem}\label{Borel}Let $(W_i)_{i\ge 1} $ be a sequence of uniformly bounded random variables such that
$$
\lim_{n\to\infty}\E(\overline S_n)=\mu,
$$
where $\overline S_n=\frac{1}{n}(W_1+\dots+W_n)$.
If\/ $\V(\overline S_n)\ll n^{-\delta}$ for some $\delta>0$, then
$$
\lim_{n\to \infty}\overline S_n=\mu
$$
almost surely.
\end{lem}
\begin{proof}
  Let $k$ be a positive integer such that $\delta k>2$. Then, by  the Chebyshev inequality,
$$
\sum_{m}\P(|\overline S_{m^k}-\E(\overline S_{m^k})|\ge 1/\sqrt m)\le
\sum_m \frac{\V(\overline S_{m^k})}{1/m}\ll \sum_m\frac 1{m^{\delta k-1}}<\infty,
$$
and so, by the Borel--Cantelli lemma,
$$
|\overline S_{m^k}-\E(\overline S_{m^k})|\ll \frac{1}{\sqrt m}
$$
almost surely.

Given $n$, let $m$ be such that $m^k\le n<(m+1)^k$. We have that $\overline S_n=\overline S_{m^k}+O(1/m)$. Then,
\begin{align*}
|\overline S_n- \mu|&\le |\overline S_n-\overline S_{m^k}|+|\overline S_{m^k}-\E(\overline S_{m^k})|+ |\E(\overline S_{m^k})-\mu|
\\
&\ll \frac{1}{m}+ \frac {1}{\sqrt m}+|\E(\overline S_{m^k})-\mu|\to 0
\end{align*} when $n\to \infty$, almost surely.
\end{proof}

\section{Proof of Theorem \ref{Sk}}\label{sec4}
\subsection{Notation and auxiliar results}

Some notation and some preliminary results are needed.
We will write
$$
\mathbf s=(s_0,s_1,\dots,s_{k-1}),
$$
with $s_0=(0,0)$ and, for $i=1,\dots,k-1$,
$$
s_i=s_{i-1}+\left\{\begin{array}{cc}
(1,0)&\text{or}\\(0,1)&
\end{array}\right.
$$
for a sequence of $k$ consecutive points of the random walk starting at $(0,0)$. Denote by $\mathcal{S}_{k-1}$ the set comprising the $2^{k-1}$ possible sequences $\mathbf{s}$. Observe that, in an $\alpha$-random walk, a sequence $\mathbf{s}$ has probability
$$
 \P (\mathbf s)=\alpha^{r(\mathbf s)}(1-\alpha)^{u(\mathbf s)},
$$
where $r(\mathbf s)$ is the number of steps $(1,0)$ in $\mathbf s$, and $u(\mathbf s)$, the number of steps $(0,1)$. Notice that $r(\mathbf s)+u(\mathbf{s})=k-1$.

\smallskip
Fix $\mathbf s=(s_0,s_1,\dots,s_{k-1})\in \mathcal S_{k-1}$ and a prime $p$. Consider, within the $p^2$ classes $(x,y)$ mod $p$, the set
\begin{equation}\label{def of Bp(s)}
\mathcal B_p(\mathbf s)=\{ (x,y) \text{ mod $p$} \,:\, (x,y)\equiv -s_i \text{ for some $i=0,1,\dots,k-1$}\},
\end{equation}
and let
\begin{equation}\label{def of Cp(s)}
\mathcal C_p(\mathbf s)=\{(x,y) \text{ mod $p$} \,:\, (x,y)\not\equiv -s_i \text{ for all $i=0,1,\dots,k-1$}\}
\end{equation}
be its complement. Observe that $
|\mathcal B_p(\mathbf s)|+|\mathcal C_p(\mathbf s)|=p^2$ and notice that, if $p\ge k$, then
\begin{equation}\label{size of Bp(s) for p bigger than k}
|\mathcal B_p(\mathbf s)|=k,
\end{equation}
because in this case all the $-s_i$ belong to different classes mod $p$.

For a positive integer $m$, writing $D_m=\prod_{p<m}p$, we also consider
$$
\mathcal{A}_{m}(\mathbf s)=\{(x,y)\text{ mod } D_m : (x,y)\in \mathcal C_p(\mathbf s) \text{ for all prime~$p< m$.}\}
$$
Observe that, by the Chinese remainder theorem,
\begin{equation}\label{size of ADm}
|\mathcal{A}_{m}(\mathbf s)|=\prod_{p<m}|\mathcal C_p(\mathbf s)|.
\end{equation}

We define now a couple of notions of (partial) visibility.

\begin{defn}
{\upshape We say that a lattice point $Q=(a,b)$ is \emph{$p$-visible} if $p\nmid \gcd(a,b).$
We say that $Q=(a,b)$ is \emph{visible at level $m$} if $P$ is $p$-visible for any prime $p< m$.}
\end{defn}

Notice that $P$ is visible if $P$ is $p$-visible por all prime $p$, or equivalently, if $P$ is visible at any level $m$.

Given a sequence $\mathbf s=(s_0,s_1,\dots,s_{k-1})$ and a lattice point $P$, write
$$
\mathbf s(P)=(P+s_0,P+s_1,\dots,P+s_{k-1}),
$$
for a sequence of $k$ consecutive points in the random walk starting at $P$.

\begin{defn}
{\upshape
We say that $\mathbf s(P)$ is \emph{$p$-visible} if the points $P+s_0,P+s_1,\dots,P+s_{k-1}$ are $p$-visible.
We say that $\mathbf s(P)$ is \emph{visible at level $m$} if $\mathbf s(P)$ is $p$-visible for all prime $p< m$.}
\end{defn}

\begin{prop}
\label{trivial}
Let $D_m=\prod_{p<m}p$. Then $\mathbf s(P)$ is visible at level $m$ if $P\equiv (x,y)$ mod $D_m$ for some $(x,y)\in \mathcal{A}_{m}(\mathbf s)$.
\end{prop}

\begin{proof}
Observe that $\mathbf s(P)$ is $p$-visible if the class of $P$ mod $p$ belongs to $\mathcal C_p(\mathbf s)$. Thus,  $\mathbf s(P)$ is $p$-visible for all prime $p<m$ if
the class of $P$ mod $p$ belongs to $\mathcal C_p(\mathbf s)$ for all $p<m$. In other words, if  the class of $P$ mod $D_m$ belongs to $\mathcal{A}_{m}(\mathbf s)$.
\end{proof}

Using the above notations, we prove that:
\begin{lem}\label{lemma:promedio ck} For $\alpha\in(0,1)$, $k\ge 1$ and a positive integer $m$,
$$
\frac 1{D_m^2}\sum_{\mathbf s\in \mathcal S_{k-1}} \P (\mathbf s)\, |\mathcal{A}_{m}(\mathbf s)|=c_k(m;\alpha),
$$
where
$$
c_k(m;\alpha)=b_k(\alpha)\prod_{k\le p<m}\Big(1-\frac k{p^2}\Big)
$$	
with $b_1(\alpha)=b_2(\alpha)=1$ and
$$
b_k(\alpha)=\sum_{\mathbf s\in \mathcal S_{k-1}} \P (\mathbf s)\prod_{p<k}\Big(1-\frac{|\mathcal B_p(\mathbf s)|}{p^2}\Big)\quad\text{for $k\ge 3$.}
$$
The function $b_k(\alpha)$ is a polynomial in $\alpha$ of degree $2\lfloor (k-1)/2\rfloor$ with rational coefficients.
\end{lem}

\begin{proof}
By \eqref{size of ADm} and \eqref{size of Bp(s) for p bigger than k},
$$
|\mathcal{A}_{m}(\mathbf s)|=\prod_{p<m}|\mathcal C_p(\mathbf s)|=\prod_{p<m}\left (p^2-|\mathcal B_p(\mathbf s)|\right )=\prod_{p<k}\left (p^2-|\mathcal B_p(\mathbf s)|)\right )\prod_{k\le p<m}\left (p^2-k\right )
$$
(we understand that the first product in the last term of the above chain of identities is equal to 1 for $k=1,2$). This yields
$$
\frac 1{D_m^2}\sum_{\mathbf s\in \mathcal S_{k-1}} \P (\mathbf s)|\mathcal{A}_{m}(\mathbf s)|=\prod_{k\le p<m}\Big(1-\frac{k}{p^2}\Big)\ \cdot \sum_{\mathbf s\in \mathcal S_{k-1}} \P (\mathbf s)\prod_{p<k}\Big(1-\frac{|\mathcal B_p(\mathbf s)|}{p^2}\Big)
.
$$

Denote by $\overline{\mathbf s}$ the complementary sequence of $\mathbf s$, that is, the result of replacing each step $(1,0)$ by $(0,1)$, and viceversa. Observe that $ \P (\overline{\mathbf s})=\alpha^{u(\mathbf s)}(1-\alpha)^{r(\mathbf s)}$ and  $|\mathcal B_p(\overline{\mathbf s})|=|\mathcal B_p(\mathbf s)|$. Thus,
\begin{align*}
b_k(\alpha)&=\sum_{\mathbf s\in \mathcal S_{k-1}} \P (\mathbf s)\prod_{p<k}\Big (1-\frac{|\mathcal B_p(\mathbf s)|}{p^2}\Big )=\frac 12 \sum_{\mathbf s\in \mathcal S_{k-1}}\big( \P (\mathbf s)+ \P (\overline{\mathbf s})\big)\prod_{p<k}\Big (1-\frac{|\mathcal B_p(\mathbf s)|}{p^2}\Big )
\\
&=\frac 12 \sum_{\mathbf s\in \mathcal S_{k-1}}\big( \alpha^{r(\mathbf s)}(1-\alpha)^{u(\mathbf s)}+\alpha^{u(\mathbf s)}(1-\alpha)^{r(\mathbf s)}\big)\prod_{p<k}\Big (1-\frac{|\mathcal B_p(\mathbf s)|}{p^2}\Big )\,.
\end{align*}
As $r(\mathbf s)+u(\mathbf s)=k-1$, it turns out that $b_k(\alpha)$ is a polynomial in $\alpha$ of degree $2\lfloor (k-1)/2\rfloor$. \end{proof}

As an example, consider the case $k=3$. The four possible sequences are
\begin{picture}(0,0)(140,72)\resizebox{5.5cm}{!}{\includegraphics{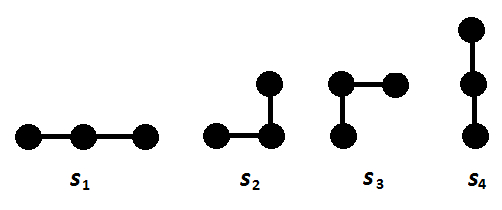}}\end{picture}
\begin{align*}
\mathbf{s}_1&=((0,0), (1,0), (2,0)),\\
\mathbf{s}_2&=((0,0), (1,0), (1,1)),\\
\mathbf{s}_3&=((0,0), (0,1), (1,1)),\\
\mathbf{s}_4&=((0,0), (0,1), (0,2)),\mbox{\hphantom{pepepepepepepepepepepepepepeppepepe}}
\end{align*}
with probabilities $\alpha^2$, $\alpha(1-\alpha)$, $\alpha(1-\alpha)$ and $(1-\alpha)^2$, respectively. Observe that, for instance,
 $$
 -\mathbf{s}_1 \text{ mod $2$} =((0,0), (1,0), (0,0))\quad\text{while}\quad
  -\mathbf{s}_2 \text{ mod $2$} =((0,0), (1,0), (1,1)).
 $$
Considering the four cases, we get $|\mathcal{B}_2(\mathbf{s_1})|=|\mathcal{B}_2(\mathbf{s_4})|=2$ and \mbox{$|\mathcal{B}_2(\mathbf{s_2})|=|\mathcal{B}_2(\mathbf{s_3})|=1$}, and consequently,
$$
b_3(\alpha)=\alpha^2\,\Big(1-\frac{2}{4}\Big)+2\alpha(1-\alpha)\,\Big(1-\frac{1}{4}\Big)+(1-\alpha)^2\, \Big(1-\frac{2}{4}\Big)=\frac{1+\alpha-\alpha^2}{3}.
$$

\subsection{Visible points at level $m$ and the proof of Theorem \ref{Sk}}

Let $X_i(m)$ be the random variable defined by $X_i(m)=1$ if $P_i$ is visible at level $m$ and $0$ otherwise.

For $k\ge 1$, define $Y_i(m)=X_i(m)\cdots X_{i+k-1}(m)$ and
$$
\overline S_{n,k}(m)=\frac{Y_1(m)+\cdots +Y_n(m)}n.
$$
Recall, from the introduction, the variables $X_i$ defined by $X_i=1$ if $P_i$ is visible and $0$ otherwise, and consider the corresponding variables $Y_i$ and $\overline S_{n,k}$.

Next we show how to deduce Theorem \ref{Sk} from Theorem \ref{SkM} below. We postpone the proof of Theorem \ref{SkM} to Section \ref{sec:proof lim ck(m)}.
\begin{thm}\label{SkM}For all $m\in\mathbb{N}$,
	$$
\lim_{n\to \infty}\overline S_{n,k}(m)=c_k(m;\alpha)\quad\text{almost surely,}
	$$
where $c_k(m;\alpha)$ was defined in Lemma {\rm \ref{lemma:promedio ck}}.
\end{thm}

Now define $Z_i(m)=X_i(m)-X_i$. Observe that $Z_i(m)=1$ if $P_i$ is visible at level~$m$, but $P_i$ is not visible; and $Z_i(m)=0$ otherwise. Consider
$$
\overline Z_{n}(m)=\frac{Z_1(m)+\cdots +Z_n(m)}n.
$$
\begin{cor}	\label{ZnM}
$$
\lim_{n\to \infty}\overline Z_{n}(m)=\prod_{p< m}\Big(1-\frac 1{p^2}\Big)-\frac6{\pi^2}
$$
almost surely.
\end{cor}

\begin{proof}Write
$$
\frac{Z_1(m)+\cdots+Z_n(m)}n=\frac{X_1(m)+\cdots+X_n(m)}n-\frac{X_1+\cdots+X_n}n.
$$
By  Theorem \ref{SkM} with $k=1$ and Theorem \ref{main},  we have
$$
\lim_{n\to \infty}\overline Z_{n}(m)=\lim_{n\to \infty}\overline S_{n,1}(m)  - \lim_{n\to \infty}\overline S_{n,1}=\prod_{p< m}\Big(1-\frac 1{p^2}\Big)-\frac 6{\pi^2}
$$
almost surely.
\end{proof}
		
\begin{lem}\label{dd}
$$
\overline S_{n,k}(m)-k\overline Z_{n}(m)-\frac{k^2}{2n}\le \overline S_{n,k}\le \overline S_{n,k}(m).
$$
\end{lem}

\begin{proof}
It is clear that
$$
Y_i(m)-Y_i=X_i(m)\cdots X_{i+k-1}(m)-X_i\cdots X_{i+k-1}
$$
takes the values $0$ or $1$. In fact, writing
$$
Y_i(m)-Y_i=(X_i+Z_i(m))\cdots (X_{i+k-1}+Z_{i+k-1}(m) )-X_i\cdots X_{i+k-1},
$$
we observe that if	$	Y_i(m)-Y_i=1$ then $Z_{i+r}(m)=1$ for some $r=0,\dots,k-1$. Thus,
$0\le Y_i(m)-Y_i\le Z_i(m)+\cdots +Z_{i+k-1}(m).$

Summing  up these inequalities for $i=1,\dots,n$ we get
$$
0\le \sum_{i=1}^nY_i(m)-\sum_{i=1}^nY_i\le \sum_{i=1}^n\sum_{r=0}^{k-1}Z_{i+r}(m)\le k\sum_{i=1}^nZ_i(m)+\sum_{i=1}^{k-1} i\le
k\sum_{i=1}^nZ_i(m)+\frac{k^2}{2}.
$$
Thus,
\begin{equation*}
0\le \overline S_{n,k}(m)-\overline S_{n,k}\le k\overline Z_{n}(m)-\frac{k^2}{2n}.\qedhere
\end{equation*}
\end{proof}

\begin{proof}[Proof of Theorem {\rm \ref{Sk}}] Taking limits in Lemma \ref{dd}, and using Theorem~\ref{SkM} and Corollary \ref{ZnM}, we get
$$
c_k(m;\alpha)-k\Big(\prod_{p< m}\Big(1-\frac 1{p^2}\Big)-\frac 6{\pi^2}\Big)\le \liminf_{n\to \infty}\overline S_{n,k}\le \limsup_{n\to \infty}\overline S_{n,k}\le c_k(m;\alpha)
$$
almost surely.
Since this is true for any $m$, we can take the limit as $m\to \infty$. Since $\lim_{m\to \infty}\prod_{p<m}(1-1/{p^2})=6/\pi^2$, we have that
$$
\lim_{n\to \infty}\overline S_{n,k}=\lim_{m\to \infty}c_k(m;\alpha)=c_k(\alpha)=b_k(\alpha)\prod_{p}\Big(1-\frac k{p^2}\Big)
$$
almost surely. This  finishes the proof.\end{proof}

\subsection{Proof of Theorem \ref{SkM}}\label{sec:proof lim ck(m)}

Theorem \ref{SkM} will follow from the estimates for $\E(\overline{S}_{n,k}(m))$ and $\V(\overline S_{n,k}(m))$ contained in Propositions \ref{E} and~\ref{Var} below, plus an extra application of Lemma~\ref{Borel}.

\begin{prop}\label{E}
$$
\E(\overline{S}_{n,k}(m))=c_k(m;\alpha)+O\Big(\frac{D_m^2}{\sqrt{ n}}\Big).
$$
	\end{prop}
\begin{proof}
	A sequence of $k-1$ lattice points $P_i,\dots ,P_{i+k-1}$ in a $\alpha$-random walk is of the form
$$
\mathbf s(P_i)=(P_i+s_0,P_i+s_1,\dots,P_i+s_{k-1}),
$$
with $P_i=(l,i-l)$  for some $0\le l\le i$, and for some $\mathbf s:=(s_0,s_1,\dots,s_{k-1})\in \mathcal S_{k-1}$.

We observe that $Y_i(m)=1$ if $X_{i+r}(m)$ is visible at level $m$ for all $r=0,\dots,k-1$.
 Thus,
\begin{align*}
\E(Y_i(m))&= \P (X_{i+r}(m)\text{ is visible at level } m \text{ for all }r=0,\dots ,k-1)
\\
&=\sum_{\mathbf s\in \mathcal S_{k-1}}\sum_{l=0}^i \P (\mathbf s) \P \big (P_i=(l,i-l)\big )\unogrande_{\mathbf s(l,i-l)\text{ is visible at level }m}.
\end{align*}
Since $\gcd(l,i-l)=\gcd(i,l)$, we have that $\mathbf s(l,i-l)$ is visible at level $m$ if and only if $\mathbf s(i,l)$ is visible at level $m$. Thus, by Proposition \ref{trivial} and Lemma \ref{lema1}, we have
\begin{align*}
\sum_{i=1}^n\E(Y_i(m))
&= \sum_{\mathbf s\in \mathcal S_{k-1}} \P (\mathbf s)\sum_{(x,y)\in \mathcal{A}_{m}(\mathbf s)}\sum_{\substack{1\le i\le n\\ i\equiv x \ \text{(mod $D_m$)} }}\sum_{\substack{0\le l\le i\\l\equiv y\ \text{(mod $D_m$)} }}\binom il \alpha^l(1-\alpha)^{i-l}
\\	
&=  \sum_{\mathbf s\in \mathcal S_{k-1}} \P (\mathbf s)\sum_{(x,y)\in \mathcal{A}_{m}(\mathbf s)}\sum_{\substack{1\le i\le n\\ i\equiv x\ \text{(mod $D_m$)}}}\Big(\frac 1{D_m}+O\Big(\frac 1{\sqrt{ i}}\Big)\Big)
\\ 	
&= \sum_{\mathbf s\in \mathcal S_{k-1}} \P (\mathbf s)\sum_{(x,y)\in \mathcal{A}_{m}(\mathbf s)}\left (\frac n{D_m^2}+O( \sqrt n)\right )
\\
&=
n\Big(\frac 1{D_m^2}\sum_{\mathbf s\in \mathcal S_{k-1}} \P (\mathbf s)|\mathcal{A}_{m}(\mathbf s)|\Big)+ O\left (D_m^2\sqrt n\right ).
\end{align*}
After dividing by $n$, we conclude the proof recalling Lemma \ref{lemma:promedio ck}.
\end{proof}
\begin{prop}\label{Var}$$V(\overline S_{n,k}(m))\ll \frac{D_m^4}{\sqrt n}.$$	
\end{prop}

\begin{proof}
We start by estimating
$$
\E(Y_i(m)\cdot Y_j(m))=\E(X_i(m)\cdots X_{i+k-1}(m)\cdot X_j(m)\cdots X_{j+k-1}(m)).
$$
Assume that $i<j-k$. A typical pair of sequences of points
$$
P_i,P_{i+1},\dots,P_{i+k-1}\quad\text{and}\quad P_j,P_{j+1},\dots,P_{j+k-1}
$$
are of the form $\mathbf t(P_i)$ and $\mathbf s(P_j)$, where
$\mathbf t=(t_0,t_1,\dots,t_{k-1})$ and $\mathbf s=(s_0,s_1,\dots,s_{k-1})$ belong to $\mathcal S_{k-1}$,
$P_i=(l,i-l),$ for some  $0\le l\le i$ and
$P_j=(l+u+r,j-l-u-r)$ for some $r,\ 0\le r\le j-i-(k-1)$, with $(u,v)=t_1+\cdots+t_{k-1}$. Notice that $u+v=k-1$. Notice also that the condition $i<j-k$ avoids  possible intersection between the sequences.
		
Thus we have
\begin{align*}
\E (Y_i(m)Y_j(m))=&\sum_{\mathbf t,\mathbf s\in \mathcal S_{k-1}}\sum_{l=0}^i\sum_{r=0}^{j-i-(k-1)}  \P (\mathbf t) \P (\mathbf s)
\\
&\cdot
 \P \big (P_i=(l,i-l)\big )
\\
&\cdot \P \big (P_j=(l+u+r,j-l-u-r)\,|\, P_{i+k-1}=(l+u,i-l+v)\big )
\\
&\cdot \unogrande_{\mathbf t(P_i)\text{ and }\mathbf s(P_j) \text{ are visible at level }m}.	
\end{align*}
Using again that $\mathbf t(P_i)=\mathbf t(l,i-l)$ and $\mathbf s(P_j)=\mathbf s(l+u+r,j-l-u-r)$ are visible at level $m$ if and only if $\mathbf t(i,l)$ and $\mathbf s(l+u+r,j)$ are visible at level $m$, we have
\begin{align*}
&\sum_{\substack{i,j\\ 1\le i< j-k\le n-k}}
\E (Y_i(m)Y_j(m))=\sum_{\mathbf t,\mathbf s\in \mathcal S_{k-1}}\sum_{1\le i< j-k\le n-k}\sum_{\substack{l=0}}^i\sum_{\substack{r=0}}^{j-i-(k-1)} \P (\mathbf t) \, \P (\mathbf s)
\\
&\qquad\qquad\qquad\cdot
\P \big (P_i=(l,i-l)\big )
\\
&\qquad\qquad\qquad\cdot \P \big (P_j=(l+u+r,j-l-u-r)\,|\, P_{i+k-1}=(l+u,i-l+v)\big )
\\
& \qquad\qquad \qquad\cdot \unogrande_{\mathbf t(l,i)\text{ and }\mathbf s(l+u+r,j) \text{ are visible at level }m}.
\end{align*}	

Using again Proposition \ref{trivial} we have
\begin{align*}
&\sum_{\substack{i,j\\ 1\le i<j-k\le n-k}}\E (Y_i(m)Y_j(m))
\\
&\quad=
\sum_{\mathbf t,\mathbf s\in \mathcal S_{k-1}}
\sum_{\substack{(x,y)\in \mathcal{A}_{m}(\mathbf t)\\(z,w)\in \mathcal{A}_{m}(\mathbf s)}}
\sum_{\substack{1\le i< j-k\le n-k\\i\equiv y\ \text{(mod $D_m$)}\\ j\equiv w\ \text{(mod $D_m$)}}}\sum_{\substack{l=0\\ l\equiv x\ \text{(mod $D_m$)}}}^i
\sum_{\substack{r=0\\r\equiv z-l-u\ \text{(mod $D_m$)}}}^{j-i-(k-1)} \P (\mathbf t)\, \P (\mathbf s)
\\
&\quad\qquad\qquad\cdot \P \big (P_i=(l,i-l)\big )
\\
&\quad\qquad\qquad\cdot \P \big (P_j=(l+u+r,j-l-u-r)\,|\, P_{i+k-1}=(l+u,i-l+v)\big )
\\
&\quad=
\sum_{\mathbf t,\mathbf s\in \mathcal S_{k-1}}
\sum_{\substack{(x,y)\in \mathcal{A}_{m}(\mathbf t)\\(z,w)\in \mathcal{A}_{m}(\mathbf s)}}
\ \sum_{\substack{1\le i< j-k\le n-k\\i\equiv y\ \text{(mod $D_m$)}\\ j\equiv w\ \text{(mod $D_m$)}}}
\ \sum_{\substack{l=0\\ l\equiv x\ \text{(mod $D_m$)}}}^i \P (\mathbf t) \, \P (\mathbf s)
\\
&\quad\cdot \!\!\sum_{\substack{r=0\\r\equiv z-l-u\ \text{(mod $D_m$)}}}^{j-i-(k-1)}
\binom il\alpha^l(1-\alpha)^{i-l}\binom{j-i-(k-1)}{r}\alpha^{r}(1-\alpha)^{j-i-(k-1)-r}.
\end{align*}
By Lemma \ref{lema1} we have
\begin{align*}
\sum_{\substack{r=0\\r\equiv z-l-u\ \text{(mod $D_m$)}}}^{j-i-(k-1)}&\, \binom{j-i-(k-1)}{r}\alpha^{r}(1-\alpha)^{j-i-(k-1)-r}
\\&\qquad
=\frac 1{D_m}+O\Big(\frac 1{\sqrt{(j-i-(k-1))}}\Big)	
\end{align*}
and
$$
\sum_{\substack{l=0\\ l\equiv x\ \text{(mod $D_m$)}}}^i \binom il\alpha^l(1-\alpha)^{i-l}=\frac 1{D_m}+O\Big(\frac 1{\sqrt{ i}}\Big).
$$
Thus,
\begin{align*}
&\sum_{\substack{i,j\\ 1\le i< j-k\le n-k}}\E (Y_i(m)Y_j(m))
=
\sum_{\mathbf t,\mathbf s\in \mathcal S_{k-1}}\sum_{\substack{(x,y)\in \mathcal{A}_{D_m}(\mathbf t)\\(z,w)\in \mathcal{A}_{D_m}(\mathbf s)}}
\P (\mathbf t) \, \P (\mathbf s)
\\
&\qquad \cdot \sum_{\substack{1\le i< j-k\le n-k\\i\equiv y \ \text{(mod $D_m$)}\\ j\equiv w\ \text{(mod $D_m$)}}}\Big(\frac 1{D_m}+O\Big(\frac 1{\sqrt{j-i-(k-1)}}\Big) \bigg)\Big(\frac 1{D_m}+O\Big(\frac 1{\sqrt{ i}}\Big)
\Big).
\end{align*}
Straightforward calculations show that
$$
\sum_{\substack{1\le i< j-k\le n-k\\i\equiv y\ \text{(mod $D_m$)}\\ j\equiv w\ \text{(mod $D_m$)}}}
	\Big(\frac 1{D_m}+O\Big(\frac 1{\sqrt{j-i-(k-1)}}\Big) \Big)\Big(\frac 1{D_m}+O(\frac 1{\sqrt{ i}})
	\Big)=\frac{n^2}{2D_m^4}+O(n^{3/2}).
$$
		Thus,
$$
\sum_{\substack{i,j\\ 1\le i< j-k\le n-k}}\E (Y_i(m)Y_j(m))
=\bigg(\frac 1{D_m^2}\sum_{\mathbf t\in \mathcal S_{k-1}}\sum_{\substack{(x,y)\in \mathcal{A}_{D_m}(\mathbf s)}} \P (\mathbf s)\bigg)^2 \Big(\frac{n^2}{2}+O(D_m^4n^{3/2})\Big).
$$
For the cases $|i-j|\le k$, where the two sequences may intersect, we use the trivial estimate
$$
\sum_{\substack{1\le i,j\le n\\ |j-i|\le k}}\E (Y_i(m)Y_j(m))\ll n.
$$
We conclude that
\begin{align}\nonumber
\sum_{1\le i,j\le n}\E (Y^i(m)Y^j(m))
&=
\Big(\frac 1{D_m^2}\sum_{\mathbf t\in \mathcal S_{k-1}} \P (\mathbf s) |\mathcal{A}_{m}(\mathbf s)|\Big)^2\, \big(n^2+O(D_m^4n^{3/2})\big)+O(n)
\\ \label{sum}
&= n^2c^2_k(m;\alpha) +O(D_m^4n^{3/2} ),
\end{align}
using Lemma \ref{lemma:promedio ck}. We finish the proof observing that
\begin{align*}
\V(\overline S_{n,k}(m))&=\E(\overline S^2_{n,k}(m))-\E^2(\overline S_{n,k}(m))
\\
&=\frac 1{n^2}	\sum_{1\le i,j\le n} \E(Y^i(m)Y^j(m))-  \E^2(\overline S_{n,k}(m)) =O\Big(\frac{D_m^4}{\sqrt n} \Big),
\end{align*}
by \eqref{sum} and Proposition \ref{E}.
\end{proof}

\section{Rational walks}\label{sec:rational}
Consider, for any real number $x\in [0,1)$ expressed in binary form, the associated infinite walk, starting at $P_0=(0,0)$, in which each $1$ in the binary representation means step $(1,0)$, and each $0$, step $(0,1)$. Define
$$
\overline{S}_n(x)=\frac{X_1+\cdots +X_n}n,
$$
where $X_i=1$ if the the  lattice point $P_i$ in this walk is visible, and $X_i=0$ otherwise.

We are interested in the limit of $\overline{S}_n(x)$ when $x$ is a rational number. It is instructive to consider an example.

The walk associated to the rational number $x=0\mbox{.}11\overline{101}$ comprises the lattice points $$\boldsymbol{(0,0),(1,0),(2,0)},(3,0)(3,1),(4,1),\dots ,(2k+1,k-1),(2k+1,k)(2k+2,k),\dots $$
The lattice points in bold correspond to the aperiodic part. The remaining points can be classified in three classes. The following facts are easy to check:

The lattice point $(2k+1,k-1)$ is visible if and only $k\not \equiv 1\pmod 3$. Thus the relative density for this class of lattice points is $\delta_1=2/3$.

The lattice point $(2k+1,k)$ is visible for all $k$. Thus, $\delta_2=1$.

The lattice point $(2k+2,k)$ is visible if only $k\not \equiv 0 \pmod 2$. Thus, $\delta_3=1/2$.

So the global density, i.e., the average of the relative densities, is
$$\lim_{n\to \infty}S_n(x)=\frac{\delta_1+\delta_2+\delta_3}3=\frac{13}{18}.$$

\begin{proof}[Proof of Theorem {\rm \ref{main2}}]
	Suppose that
	$$
	x=0\mbox{.}\,\alpha_1\cdots \alpha_m \,\overline{\alpha_{m+1}\cdots\alpha_{m+l}},
	$$
	with $m\ge 0$ and a period of length $l\ge 1$. Write $v_i$ for the step $(1,0)$ or $(0,1)$ associated with each $\alpha_i$. Denote by $(x_0,y_0)$ the lattice point reached after the steps corresponding to the the nonperiodic part, and write, for $i=1,\dots,l$, $(r_i,t_i)=v_{m+1}+\cdots +v_{m+i}$.
	The point $(r_l,t_l)$ will be simply denoted by $(r,t)$.
	
	For $n\ge 1$, we display the first $ln$ lattice points (after the nonperiodic part) in the following table:
	$$
	\begin{array}{l|ll|l}
	(x_0,y_0)+(r_1,t_1) &\cdots &\cdots &(x_0,y_0)+(r_l,t_l)
	\\
	(x_0,y_0)+(r_1,t_1)+(r,t)& \cdots & \cdots& (x_0,y_0)+(r_l,t_l)+(r,t)
	\\
	\qquad\vdots &                &
	\\
	(x_0,y_0)+(r_1,t_1)+(n-1)(r,t)& \cdots & \cdots& (x_0,y_0)+(r_l,t_l)+(n-1)(r,t).
	\end{array}
	$$
	
Fix now $i$, $1\le i\le l$, and define the number
\begin{equation}\label{def of mi}
m_i:=t(x_0+r_i)-r(y_0+t_i)\,.
\end{equation}
Observe that $m_i$ is determined by $x$.

Consider the $n$ lattice points $P_0, P_1,\dots,P_{n-1}$ located at column $i$ in the previous table, of the form
$$
P_k=(x_0+r_i+kr\,,\,y_0+t_i+kt),\quad 0\le k\le n-1.
$$

We count now the number of $P_k$ that are visible.

If $m_i=0$, then $t\,(x_0+r_i)=r\,(y_0+t_i)$, and thus each point $P_k$ can be written as
$$
P_k=\Big(x_0+r_i+kr\,,\,\frac{t}{r}(x_0+r_i+kr)\Big)=\Big(\frac{x_0+r_i+kr}{r/\gcd(r,t)}\, \frac{r}{\gcd(r,t)}\,,\,\frac{x_0+r_i+kr}{r/\gcd(r,t)}\, \frac{t}{\gcd(r,t)}\Big)
$$
Since $\frac{x_0+r_i+kr}{r/\gcd(r,t)}>1$ for $k\ge 1$, this shows that each $P_k$, $k\ge 1$, is non visible in this case, and so the relative density of visible points is $\delta_i=0$ or $\delta_i=1/n$.

For $m_i\ne 0$, we write the number of visible points $P_k$ as
$$\sum_{0\le k\le n-1}\uno_{\gcd(x_0+r_i+kr,y_0+t_i+kt)=1}=\sum_{0\le k\le n-1}\ \sum_{d\mid \gcd(x_0+r_i+kr,y_0+t_i+kt)}\mu(d).$$

Notice that if $d\mid (x_0+r_i+kr)$ and $d\mid (y_0+t_i+kt)$, then $d\mid m_i$, so we can write the sum above as
\begin{align*}
\sum_{d\mid m_i}\mu(d)\, &\#\{k: 0\le k\le n-1,\ x_0+r_i+kr\equiv y_0+t_i+kt\equiv 0\pmod d\}
\\
	&\qquad
=\sum_{d\mid m_i}\mu(d)\Big(\frac ndS_i(d)+O(d)\Big)=n\sum_{d\mid m_i}\frac{\mu(d)S_i(d)}d+O\Big(\sum_{d\mid m_i} d\Big),
\end{align*}
where $S_i(d)$ is the number of solutions $k \pmod d$ of
$$x_0+r_i+kr\equiv y_0+t_i+kt\equiv 0\pmod d.$$
By the Chinese remainder theorem, the function $S_i(d)$ is a multiplicative function, and then the relative density $\delta_i$ is
$$\delta_i=\sum_{d\mid m_i}\frac{\mu(d)S_i(d)}d=\prod_{p\mid m_i}\left (1-\frac{S_i(p)}p\right ).$$

Fix $p\mid m_i$.

If $p\nmid r$ and $p\nmid t$ then $S_i(p)=1$. To see this we observe that $x_0+r_i+kr\equiv 0\pmod p\iff k\equiv -(x_0+r_i)r^{-1}\pmod p$, and that $y_0+t_i+kt\equiv 0\pmod p\iff k\equiv -(y_0+t_i)t^{-1}\pmod p$. Since $p\mid m_i$, we have that indeed $-(x_0+r_i)r^{-1}\equiv -(y_0+t_i)t^{-1}\pmod p$.

If $p\mid r$ and $p\nmid t$,  then $p\mid (x_0+r_i).$ This implies that $x_0+r_i+kr\equiv 0\pmod p$ for all $k$ and then  the lattice point is visible when $k\equiv (-y_0+t_i)t^{-1}\pmod p$. Thus, $S_i(p)=1$ in this case.
If $p\mid t$ and $p\nmid r$ we also have that $S_i(p)=1$.

Finally, if $p\mid r$ and $p\mid t$, then $S_i(p)=0$ except when $p\mid (x_0+r_i)$ and $p\mid (y_0+t_i)$ simultaneously. In this case we have $S_i(p)=p$.

Upon considering the cases above, we get that
\begin{equation}
\label{deltai}
\delta_i=\prod_{p\mid m_i}\left (1-\frac 1p\right )\prod_{p\mid \gcd(m_i,r,t)}\frac{\epsilon_i(p)}{(1-1/p)},
\end{equation}
where $\epsilon_i(p)=1$ if $p\nmid \gcd(m_i,r,t,x_0+r_i,y_0+t_i)$, and $\epsilon_i(p)=0$ otherwise.


Thus, the number of $0\le k\le n-1$ satisfying the condition above is
	$$
	n\delta_i+O(1).
	$$
	Summing up for each $i=1,\dots ,l$, we have that
	$$\overline S_{ln}(x)=\frac{1}l\sum_{i=1}^l\delta_i+O(1/n).$$
	Since for all $m=0,\dots,l-1$ we have that $$\overline S_{ln}(x)\frac{ln}{ln+m}\le \overline S_{ln+m}(x)\le \overline S_{ln}(x)\frac{ln}{ln+m}+\frac m{ln+m}$$ we conclude that
	\begin{equation}\label{proportion}
		\lim_{n\to \infty}\overline S_n(x)=\frac 1l\sum_{i=1}^l\delta_i,
	\end{equation}
	that is clearly a rational number.
\end{proof}

The proof above contains a procedure for calculating the proportion of visible points visited by the walk associated to a rational number $x$. Consider, as an illustration, a periodic part of the form $\overline{0110}$. This means, in the notation of the proof, that $r=t=2$. 

Consider, for example, the rational number $x=0\mbox{.}1000\,\overline{0110}$, in which the aperiodic part leads to the lattice point $(x_0, y_0)=(1, 3)$. In this case the $|m_i|$'s are $6$, $4$, $2$ and $4$. Then the proportion is, following  \eqref{deltai} and \eqref{proportion},
$$
\frac{1}{4}\Big[\Big(1-\frac{1}{3}\Big)+0+1+1\Big]=\frac{2}{3}.
$$
Notice that no factors $(1-1/2)$ do appear, because $r$, $t$ and the $m_i$'s are even numbers, and that the term $0$ appears because $\epsilon_2(2)=0$ in this case.

Compare with the case $x=0\mbox{.}10000\,\overline{0110}$, that corresponds to $(x_0, y_0)=(1, 4)$, and for which the $|m_i|$'s are $8$, $6$, $4$ and $6$. Now the proportion is
$$
\frac{1}{4}\Big[1+\Big(1-\frac{1}{3}\Big)+1+\Big(1-\frac{1}{3}\Big)\Big]=\frac{5}{6}.
$$
Or with the number $x=0\mbox{.}1000\,\overline{0111}$, with periodic part corresponding to $r=3$ and $t=1$, and aperiodic part leading to the point $(x_0,y_0)=(1,3)$. Now the $|m_i|$'s are $11$, $10$, $9$ and $8$, and the proportion turns out to be
$$
\frac{1}{4}\Big[\Big(1-\frac{1}{11}\Big)+\Big(1-\frac{1}{2}\Big)\Big(1-\frac{1}{5}\Big)+\Big(1-\frac{1}{3}\Big)+\Big(1-\frac{1}{2}\Big)\Big]=\frac{817}{1320}.
$$

\end{document}